\documentclass[11pt]{article} 
\usepackage{amsmath,amsfonts,amssymb,graphicx,amsthm,color,natbib,hyperref}
\usepackage{authblk}
\usepackage[normalem]{ulem}
\allowdisplaybreaks
\usepackage[]{geometry}

\newtheorem{proposition}{Proposition}
\theoremstyle{definition}

\newtheorem{example}{Example}
\newtheorem{remark}{Remark}

\title{Estimation and goodness-of-fit testing for non-negative random variables with explicit Laplace transform}

\begin{document}

\author[1]{{Lucio} {Barabesi}}
\author[2,3]{{Antonio} {Di Noia}}
\author[1]{{Marzia} {Marcheselli}}
\author[1]{{Caterina} {Pisani}}
\author[4]{{Luca} {Pratelli}}
\affil[1]{Department of Economics and Statistics, University of Siena}
\affil[2]{Seminar for Statistics, Department of Mathematics, ETH Zurich}
\affil[3]{Faculty of Economics, Euler Institute, Università della Svizzera italiana}
\affil[4]{Italian Naval Academy}
\date{\today}
\setcounter{Maxaffil}{0}
\renewcommand\Affilfont{\itshape\small}
\maketitle
\let\thefootnote\relax\footnotetext{\emph{Email addresses:} lucio.barabesi(\sout{at})unisi.it (Lucio Barabesi), antonio.dinoia(\sout{at})stat.math.ethz.ch (Antonio Di Noia), marzia.marcheselli(\sout{at})unisi.it (Marzia Marcheselli), caterina.pisani(\sout{at})unisi.it (Caterina Pisani), luca\_pratelli(\sout{at})marina.difesa.it (Luca Pratelli).}
	
	\begin{abstract}
		Many flexible families of positive random variables exhibit non-closed forms of the density and distribution functions and this feature is considered unappealing for modelling purposes. However, such families are often characterized by a simple expression of the corresponding Laplace transform. Relying on the Laplace transform, we propose to carry out parameter estimation and goodness-of-fit testing for a general class of non-standard laws. We suggest a novel data-driven inferential technique, providing parameter estimators and goodness-of-fit tests, whose large-sample properties are derived. The implementation of the method is specifically considered for the positive stable and Tweedie distributions. A Monte Carlo study shows good finite-sample performance of the proposed technique for such laws.
	\end{abstract}
	

 \section{Introduction}\label{SecIntro}
	Large classes of positive random variables display a simple closed form of the Laplace transform, even if their density functions can be solely given by means of special functions or series, which eventually require rather complex algorithms for their computation. Owing to this shortcoming, field scientists often discard the use of such random variables, although they are appropriate for modelling real data.
	
	An archetype of such a class of laws is the positive stable distribution, which may be very suitable to model data with Paretian tails (see e.g. \citealp{NOLAN}). The density function of a positive stable random variable can be expressed by means of the Wright function (see e.g.\ \citealp{BAR}), which unfortunately is awkward to compute (see the algorithms proposed by \citealp{LUCHKO}). To this aim, \cite{BAR}, \cite{DUNN2005} and \cite{DUNN2008} suggest some approximation methods based on {\em ad-hoc} Fourier or Laplace inversion techniques. However, these algorithms can be time-consuming and possibly inadequate for evaluating the maximum-likelihood estimates, especially when large dataset are at disposal. In this case, the maximum-likelihood method could be even prohibitive and alternative techniques are required.
	
	Among these classes, the Tweedie distribution, owing to its flexibility, is useful for modelling data arising from a plethora of different frameworks (see e.g.\ \citealp{BARCEPE}, \citealp{DUNN2005}, \citealp{TWEE}).
	As a matter of fact, this family encompasses moderate heavy-tailed distributions, as well as light-tailed distributions and it also comprises the positive stable distribution as a special case (for more details, see the book of \citealp{NOLAN}). Actually, the Tweedie distribution is a tempered positive stable distribution. In addition, the Tweedie law may even model data with structural zeroes, since for some parameter ranges it is the mixture of the Dirac mass at zero and an absolutely-continuous positive distribution (see \citealp{AAL}). This feature is especially appealing when dealing with data arising from socio-economic or environmental phenomena (see \citealp{BARCEPE}, \citealp{HADU}), which indeed may produce structural zeroes. However, the distribution of a Tweedie random variable in turn involves the Wright function, giving rise to computational drawbacks.
	
	In general, a large body of literature has been devoted to the tempering of heavy-tailed laws (see e.g. the monograph by \citealp{GR}). Indeed, even if heavy-tailed distributions are well-motivated models in a probabilistic setting, extremely fat tails may be unrealistic for many real applications. Such a drawback has led to the introduction of models which are morphologically similar to the original distributions, even if they display lighter tails. In this setting, large classes of tempered distributions may be formulated as scale mixtures of a Tweedie random variable with a mixturing positive random variable (see \citealp{BARCE}, \citealp{TORRI}).  However, the corresponding density functions involve the generalized Mittag-Leffler function, which is even more challenging to compute that the Wright function (see \citealp{BAR} and references therein). In turn, the use of maximum-likelihood methods could be unfeasible for these models. Moreover, even least-square methods based on the theoretical and empirical Laplace or Fourier transforms could be inadequate, since they are likely to produce a criterion function which is not easily manageable, and possibly with multiple local minima. In addition, computational effort could be prohibitive for large datasets.
	
	In this paper, we propose a suitable class of parameter estimators for positive random variables showing a simple Laplace transform.  The proposed technique has connections with the procedures based on the probability generating function for integer-valued random variables by \cite{DINOIA} and \cite{di2023censoring}. The consistency and the large-sample distribution of the suggested estimators are obtained, as well as estimators of their asymptotic variance, thus naturally allowing to introduce appropriate goodness-of-fit (GOF) test statistics. GOF tests are especially welcome for the considered class of models, as no such proposals are present in statistical literature.
	
	The paper is organized as follows. In Section \ref{SecPreli} some preliminaries and background remarks are given. In Section \ref{SecCenso} the censoring strategy is described and the estimation procedure is introduced, as well as its asymptotic properties are derived. In Sections \ref{SecPS} and \ref{SecTW} the methodology is respectively adapted to the positive stable and Tweedie laws, and specific GOF tests are proposed. Section \ref{SecSimu} is devoted to numerical experiments. Finally, conclusions are drawn in Section \ref{discu}.

	\section{Preliminaries}\label{SecPreli}
	Let us consider a non-negative random variable (r.v.) $X$ defined on the probability space $(\Omega,{\mathcal A},P)$ and let $L_X$ be the Laplace transform of $X$, i.e.\
	\begin{gather*}
	L_X(s)=E(e^{-sX})
	\end{gather*}
	for $s\geq 0$. Obviously, if $F$ is the cumulative distribution function of $X$, it holds $L_X(s)=\int_0^1F(-\log t/s)\, dt$ for $s>0$. Moreover, if $X$ is an absolutely-continuous random variable, we have $L_X(t)=\int_0^\infty e^{-sx}\, F^\prime(x)\, dx.$ A large distribution family has a Laplace transform of type
	\begin{gather}\label{class1}
	L_X(s)=g(s^\gamma),
	\end{gather}
	where $\gamma$ is a parameter defined on a suitable subset of $\,]0,1]$, while $g$ is an appropriate function (see e.g.\ \citealp{BAR}, \citealp{TORRI}, and references therein). 
	\begin{example}\label{Jacobi}
	The family \eqref{class1} includes some rather exotic distributions, such as the generalized Jacobi laws, which have Laplace transforms given by	\begin{gather*}\label{jacobi1}
	L_X(s)={{s^\gamma}\over{{\rm sinh}(s^\gamma)}}
	\end{gather*}
	and
	\begin{gather}\label{jacobi2}
	L_X(s)={{1}\over{{\rm cosh}(s^\gamma)}},
	\end{gather}
	\noindent where $\gamma\in\,]0,1/2]$ (see \citealp{BIPITY} and  \citealp{DE2009b}). 
	\end{example}
	\begin{example}\label{ExaPS}
	The family with Laplace transform given by \eqref{class1} encompasses the positive stable distribution. Indeed, a positive stable r.v. with parameters $(\gamma,\lambda)$, denoted by ${\rm PS}(\gamma,\lambda)$, has Laplace transform
	\begin{gather}\label{LapPS}
	L_X(s)=e^{-\lambda s^\gamma},
	\end{gather}
	where $(\gamma,\lambda)\in\,]0,1]\,\times\,]0,\infty[$. For more details on the properties of this r.v.\ and its stochastic representations, see e.g.\ \cite{DEJAM}. The integer-valued counterpart of the positive stable distribution is considered in \cite{steutel1979discrete}, \cite{BARPRAT} and \cite{MARBABA}. Parameter estimation for the positive stable distribution is discussed in depth in Section \ref{SecPS}.

	It should be remarked that $g$ could be in turn the Laplace transform of a further positive (not necessarily absolutely-continuous) r.v.\ $V$, that is $g=L_V$. More precisely, if 
	\begin{gather*}
	X=V^{1/\gamma}Z,
	\end{gather*}
	where $Z$ is a positive stable r.v.\ with parameters $(\gamma,1)$ independent of $V$ (i.e.\ $X$ is a scale mixture of positive stable distributions), we have
	\begin{gather*}
	L_X(s)=E(e^{-sV^{1/\gamma}Z})=E(e^{-s^\gamma V})=L_V(s^\gamma)=g(s^\gamma).
	\end{gather*}
	Loosely speaking, in this setting the parameter $\lambda$ is assumed to be a positive  r.v.\ with a suitable distribution. The resulting family encompasses many distributions including the positive Linnik distribution which is obtained when  $V$ is distributed according to the Gamma distribution, see e.g.\ \cite{BARCE}, \cite{HUILLET2000} and \cite{JOSE}.  
	\end{example}

	\smallskip
	An even larger model family of distributions has a Laplace transform of type
	\begin{gather}\label{class2}
	L_X(s)=g(\{\theta+s\}^\gamma-\theta^\gamma),
	\end{gather}
	where $g$ is an appropriate function and $\theta\geq 0$. If $\theta>0$, the parameter $\gamma$ may also be defined on $\,]-\infty,1]$. Obviously, for $\theta=0$ the family given in \eqref{class1} is obtained.
In turn, the function $g$ in \eqref{class2} could be a Laplace transform of a further positive r.v.\ $V$. Distribution classes of such a type are considered by \cite{BARCE}, \cite{JAMES} and \cite{TORRI}, among others.\begin{example}\label{ExaTW}
	The typical member of the family with Laplace transform given by \eqref{class2} is the Tweedie distribution. A Tweedie r.v. with parameters $(\gamma, \lambda, \theta)$, denoted by TW$(\gamma, \lambda, \theta)$, has the Laplace transform
	\begin{gather}\label{LapTW}
	L_X(s)=e^{{\rm sgn}(\gamma)\lambda\{\theta^\gamma-(\theta+s)^\gamma\}},
	\end{gather}
	where $(\gamma,\lambda,\theta)\in\{[-\infty,0]\times\,]0,\infty[\times\,]0,\infty[\,\}\cup \{\,]0,1]\times\,]0,\infty[\times[0,\infty[\,\}$. The distribution has been popularized and analysed at length by \cite{HOUGAARD} and \cite{JOR}, following the seminal ideas described in \cite{TWEE}. The distribution has been often used in actuarial studies and ecology, among others (see e.g.\ \citealp{JOR} and \citealp{KENDAL2004}). For its integer-valued counterpart, see \cite{BAC} and \cite{BARPRA}. As to the random variate generation of a Tweedie r.v., see \cite{BARPRAT}, \cite{BARPRAT2}, and references therein. The family is very flexible and can be adopted as a model for datasets with complex and non-standard features, since it can simultaneously arrange fat tails and structural zeroes.
The inferential issues for the Tweedie law will be discussed in details in Section \ref{SecTW}.	\end{example}

	Albeit the described families are very appealing from a probabilistic perspective, the drawback with their use consists in the complex structure of the corresponding density and distribution functions, despite the simple forms of their Laplace transform. Even if some approximations are available in order to compute maximum-likelihood estimates (as proposed by \citealp{BAR}), their computation may be too time-consuming. Thus, alternative and simpler estimation methods may be welcome. These inferential methods should also produce GOF testing procedures which are particularly desirable, since no such proposals are present in literature.

	\section{Parameter estimation and GOF testing}\label{SecCenso} 

	Let the Laplace transform $L_X$ of the target random variable $X$  belong to the general class \eqref{class2},  specifically given by
	$L_X(s)=g(\lambda\{(\theta+s)^\gamma-\theta^\gamma\}), $
where $g$ is a suitable infinitely differentiable function on ${\mathbb R}\setminus\{0\}$, $\gamma\neq 0$ and $\lambda\kern-0.3mm>\kern-0.3mm 0$. In particular, $E(X)$ is not finite for $\theta=0$, $\lambda>0$ and $\gamma\in \, ]0,1[$. Our objective is to estimate the parameters and to obtain GOF test statistics by empirically estimating $L_X$ and its derivatives at appropriate values. 

 In this section, we introduce a methodology that is  applicable to classes satisfying \eqref{class2} and more  broadly to even more general classes.
 In the following sections, we will illustrate these general results in detail for positive stable distributions and Tweedie distributions.
 
 \smallskip  Let us start by noting that for any non-negative random variable $X$, the following relationship holds: 	\begin{gather}\label{syss}
	L_{X}^{(r)}(a)=(-1)^rE(X^re^{-aX}),\qquad \forall a>0,\ \forall r\in\mathbb N.
	\end{gather}
Even if $E(X)$ is not finite, for $a>0$ we have
	\begin{gather}\label{mom1}
E(X^re^{-aX})=E(Y^r_a)
	\end{gather} where $Y_a$ is the {\it exponentially censored random variable} $X\mathbf{1}_{\{aX<T\}},$ and $T$ is a standard exponential random variable, independent of $X$. Equation \eqref{mom1} suggests to introduce a variant of the method of moments by considering $E(X^re^{-aX})$ for a suitable $a>0$ when $L_X$ depends on some parameters to be estimated. To estimate the model parameters, a natural approach could rely on \eqref{syss} and on the empirical counterpart of $E(X^re^{-aX})$, where $r$ does not exceed the number of parameters. In practical applications, the number of parameters typically is $1$, $2$, or $3$, making this method feasible for implementation with real data.

\medskip\medskip\noindent Choosing a fixed deterministic point $a>0$ for calculating $L_X$ does not provide robust estimations of model parameters. This is because $L_X(a)$ could assume negligible values or values practically equal to one for many pairs of parameters $\lambda$ and $\gamma$, with $\theta=0$ or $\theta>0$, making any information on the parameters irrelevant. This drawback persists even when  multiple values of $a$ are fixed deterministically, and $L_X(a)$ with its empirical counterpart are considered. However, the choice of $a$, denoted by $a_*$ is crucial for this purpose. The \lq optimal' $a_*$ will depend on the model parameters, and we propose a data-driven method for obtaining a coherent estimator $A$ of $a_*$. To achieve this, first consider that $L_X([0,\infty[)=]p,1]$, where $p=P(X=0)$.

 If $p=0$, to avoid $Y_{a_*}$ being too censored or scarcely censored (i.e. $E(X^re^{-a_*X})$ being practically equal to zero or \lq infinite') and thus loosing information on parameters, we could set $a_*=L_X^{-1}(c_*)$ where $L^{-1}_X$ denotes the inverse function of $L_X$ and $c_*$ is an element of $[1/5,4/5]$. Moreover, in the setting of positive stable distribution, it is convenient to maximize $-c_*\log(c_*)$ for estimating $\gamma$. Since $c_*=1/e$ is the maximum point, in the following, we will take $$a_*=L_X^{-1}(1/e)$$ when $p=0$. Otherwise, $c_*$ will be chosen to be greater than $p$.  Now, let $X_1,\ldots,X_n$ be $n$ independent copies of the r.v.\ $X$ and let
	\begin{gather*}
	L_n(s)={{1}\over{n}} \sum_{i=1}^n e^{-sX_i}
	\end{gather*}
	be the empirical counterpart of $L_X$. Denote $F_n$ as the empirical cumulative distribution function associated with $X_1,\ldots,X_n$. It is worth noticing that for all $s>0$, it holds
	$$|L_X(s)-L_n(s)|=\Big|\int_0^1\big\{F(-\log t/s)-F_n(-\log t/s)\big\}\, dt\Big|\leq ||F-F_n||_\infty.$$
	Moreover, let $A$ be a further positive random variable, which is a function of $X_1,\ldots,X_n$ and is defined by the implicit equation 
	\begin{gather*}
	L_n(A)={{1}\over{n}} \sum_{i=1}^n e^{-AX_i}=e^{-1}.
	\end{gather*}
	when $p=0$. This implicit definition of  $A$ ensures that $A$ is data-driven, making it adaptive to the sample data $X_1,\ldots,X_n$. Since $\cup_{i=1}^n\{X_i=0\}$ is a negligible event when $p=0$, on the basis of the Glivenko-Cantelli Theorem,  $A$ is well defined and converges almost surely to $a_*=L^{-1}_X(e^{-1})$ because $$|L_X(A)-L_X(a_*)|=|L_X(A)-1/e|=|L_X(A)-L_n(A)|\leq ||F-F_n||_\infty.$$
If $p<1/e$, $A$ is well defined for large $n$ and the same asymptotic result holds true. Otherwise, if $p\geq 1/e$, we can consider $c_*=\frac {1+(e-1)p}{e}$ and define $A$ for large $n$ by the implicit equation
	\begin{gather*}
	L_n(A)={{1}\over{n}} \sum_{i=1}^n e^{-AX_i}=\frac {1+(e-1)\widehat p_n}{e}
	\end{gather*}
where $\widehat p_n={{1}\over{n}} \sum_{i=1}^n I_{\{X_i=0\}}$. For this choice of $c_*$ similar asymptotic results for $A$ are satisfied as when $c_*=1/e$.  Additionally, we denote
	\begin{gather*}
	m_{r,a_*}=E(Y_{a_*}^r)=E(X^re^{-a_*X})
	\end{gather*}
	with empirical counterpart
	\begin{gather*}
	\widehat m_{r,A}=
	{{1}\over{n}} \sum_{i=1}^n X_i^re^{-AX_i}.
	\end{gather*}
In the following Proposition we obtain the large-sample properties of the r.v.s $A$ and $\widehat m_{1,A},\ldots,\widehat m_{k,A}$, with $k\geq 1$, by means of a suitable use of the Delta Method.

	\begin{proposition}\label{ProGen}
	If $p<1/e$ then the random vector
	$(\widehat m_{1,A},\ldots,\widehat m_{k,A},A)$ converges almost surely to the vector $(m_{1,a_*},\ldots,m_{k,a_*},a_*)$, while the random vector $\sqrt n(\widehat m_{1,A}-m_{1,a_*},\ldots,\widehat m_{k,A}-m_{k,a_*},A-a_*)$ converges in distribution to the multivariate normal law $\mathrm{N}_{k+1}(0,\Sigma)$ as $n\to\infty$. Here, $\Sigma$ is the variance-covariance matrix of the random vector $(V_1,\ldots,V_k,W)$, with 
	\begin{gather*}
		V_r=e^{-a_*X}\Big(X^r-{{m_{r+1,a_*}}\over{m_{1,a_*}}}\Big),
	\end{gather*}
	for $r=1,\ldots,k$, and $$W={{e^{-a_*X}}\over{m_{1,a_*}}}.$$ A consistent estimator of $\Sigma$ is given by the sample variance-covariance matrix of the random vectors $(\widehat V_{1,i},\ldots,\widehat V_{k,i},\widehat W_i)$, $i=1,\ldots,n$, where 
	\begin{gather*}\label{varest}
		\widehat V_{r,i}= e^{-AX_i} \Big(X^r_i-{{\widehat m_{r+1,A}}\over{\widehat m_{1,A}}}\Big),\qquad \widehat W_i={{e^{-AX_i}}\over{\widehat m_{1,A}}}.
	\end{gather*}
	\end{proposition}

	\begin{proof}
	For simplicity of notation, let us denote $a=a_*$. If $u=(u_1,\ldots,u_k)$ is an element of $\mathbb{R}^k$ and $t\in\mathbb{R}$, let us consider the r.v.\
	\begin{gather*}
		U_n=\sqrt n\Big\{\sum_{r=1}^k u_r(\widehat m_{r,A}-m_{r,a})+t(A-a)\Big\}.
	\end{gather*}
	On the basis of the Mean Value Theorem there exists a map $h$ defined on $]0,\infty[\,^2$ such that, for any $b,x>0$, it holds
	\begin{gather*}
		e^{-bx}=e^{-ax}-xe^{-h(b,x)x}(b-a),
	\end{gather*}
	with $h(b,x)\in\,]\min(b,a),\max(b,a)[$. Thus, we have 
	\begin{gather*}
		U_n=\sqrt n\Big\{\sum_{r=1}^k u_r(\widehat m_{r,a}-m_{r,a})+(A-a)(t-\sum_{r=1}^k{u_r\over n}\sum_{i=1}^nX_i^{r+1}e^{-h(A,X_i)X_i})\Big\}.
	\end{gather*}
	Since $A$ converges almost surely to $a$, the Law of Large Numbers and the Continuous Mapping Theorem imply that
	\begin{gather*}
		\lim_n{1\over n}\sum_{i=1}^nX_i^{r+1}e^{-h(A,X_i)X_i}=m_{r+1,a}
	\end{gather*}
	almost surely. Moreover, if $p=0$ there exists a negligible event $H$ such that $s\mapsto L_n(s)(\omega)$ is a strictly decreasing function for any $\omega\in H^c$.  These functions are invertible for any $\omega\in H^c$ and their inverse mappings are differentiable at $L_n(a)(\omega)$ with derivative $1/L^\prime_n(a)(\omega)$.  Then $\sqrt n(A-a)$ coincides almost surely with
	\begin{gather*}
		\sqrt n\{L^{-1}_n(e^{-1})-L^{-1}_n(L_n(a))\}\kern-0.6mm=\kern-0.6mm{{\sqrt{n}}\over{L^\prime_n(a)}}\{e^{-1}-L_n(a)\}\kern-0.5mm+\kern-0.5mm o_P(\sqrt n\{e^{-1}-L_n(a)\})
	\end{gather*}
	 Since $L^\prime_n(a)$ converges almost surely to $-m_{1,a}$, and $o_P(\sqrt n\{e^{-1}-L_n(a)\})$ coincides with $o_P(\sqrt n\{L_X(a)-L_n(a)\})=\sqrt n\{L_X(a)-L_n(a)\}o_P(1),$ from the convergence in distribution of $\sqrt n\{L_X(a)-L_n(a)\}$ it follows that $U_n$ has the same large-sample behaviour of the r.v.\
	\begin{gather*}
		S_n=\sqrt n\Big[\sum_{r=1}^k u_r(\widehat m_{r,a}-m_{r,a})+{{1}\over{m_{1,a}}}\{L_n(a)-e^{-1}\}(t-\sum_{r=1}^ku_rm_{r+1,a})\Big].
	\end{gather*} 
	 This holds true for large $n$, even when $p\in]0,1/e[$.
	Observe that
	\begin{align*}
		S_n=\,&\sqrt n\Big[\sum_{r=1}^k {{u_r}\over{n}}\sum_{i=1}^n\big\{(X_i^{r}e^{-a X_i}-m_{r,a})
		-{{m_{r+1,a}}\over{m_{1,a}}}(e^{-a X_i}-e^{-1})\big\}\\
		&+{{t}\over{nm_{1,a}}}\sum_{i=1}^n(e^{-aX_i}-e^{-1})\Big].
	\end{align*}
	Thus, the classical Central Limit Theorem ensures that $S_n$ converges in distribution to the normal law $\mathrm{N}(0,\sigma^2)$,
	where $\sigma^2={\rm Var}(\sum_{r=1}^k u_rV_r+tW).$ The result follows from the Cramér-Wold Theorem and the Law of Large Numbers. In particular, the sample variance-covariance
	matrix of $(\widehat V_{1,i},\ldots,\widehat V_{k,i},\widehat W_i)$ is a consistent estimator of $\Sigma$, since $(\widehat m_{1,A},\ldots,\widehat m_{k,A},A)$ converges almost surely to $(m_{1,a},\ldots,m_{k,a},a)$.
	\end{proof}

\begin{remark}
From Proposition \ref{ProGen}, it follows that $\sqrt n(A-a_*)$ converges in distribution to the normal law $\mathrm {N}(0,\{L_X(2a_*)-e^{-2}\}/m_{1,a_*}^2)$ and has the same large-sample behaviour as the r.v.\ $
	{{\sqrt n}\over{m_{1,a_*}}}\big({{1}\over{n}}\sum_{i=1}^n e^{-a_*X_i}-e^{-1}\big).$
\end{remark}

\begin{remark}
If $p\geq 1/e$, by taking $a_*=L_X^{-1}(\frac{1+p(e-1)}{e})$ a similar proposition can be proven, with some changes in the matrix $\Sigma$ and the random variables $V_r$ and $W$.\end{remark}

\medskip 	Let us now suppose that the distribution of $X$ depends on $k$ parameters, say $\alpha_1,\ldots,\alpha_k$. Let us also suppose that, on the basis of \eqref{syss}, there exist $k$ differentiable functions $h_1,\ldots,h_k$ and a natural number $r\leq k$, such that $\alpha_j=h_j(m_{1,a_*},\ldots,m_{r,a_*},a_*)$ for $j=1,\ldots,k$. Some $h_j$ could be solely functions of $a_*$ and in this case we take $r=0$. Thanks to Proposition \ref{ProGen}, consistent and large-sample normal estimators ${\widehat\alpha_j}$ of $\alpha_j$, for $j=1,\ldots,k$, are given by
	\begin{gather*}
	{\widehat \alpha_j}=h_j(\widehat m_{1,A},\ldots,\widehat m_{r,A},A).
	\end{gather*}
	Moreover, if $L_X(s)=L_X(s;\alpha_1,\ldots,\alpha_k)$ we also adopt the notation
	\begin{gather*}
	{\widehat L_X(s)}=L_X(s;\widehat\alpha_1,\ldots,\widehat\alpha_k).
	\end{gather*}
On the basis of the previous discussion, considering the empirical version of \eqref{syss}, a natural GOF test statistic is given by\begin{gather}\label{test1}
	T_n=\sqrt n\{\widehat m_{r+1,A}+(-1)^r\widehat L_X^{(r+1)}(A)\} ,
	\end{gather}
where $\widehat L_X^{(r)}$ represents the $r$-th derivative of $\widehat L_X$. Other GOF test statistics can be obtained by substituting $r+1$ with a natural number $q>r+1$. If $\widehat L_X(A)$ is not a degenerate random variable, a further GOF test statistic is given by
	\begin{gather}\label{test2}
	T^\prime_n=\sqrt n\{e^{-1}-\widehat L_X(A)\}.
	\end{gather}
\begin{remark}
Note $T_n$ (or $T^\prime_n$) can be written as $$\sqrt n\big\{\psi(\widehat m_{1,A},\ldots,\widehat m_{r+1,A},A)-\psi(m_{1,a_*},\ldots,m_{r+1,a_*},a_*)\big\}$$ where $\psi$ is a differentiable function. From Proposition \ref{ProGen} and the Delta Method, it follows that $T_n$ (or $T^\prime_n$) converges in distribution to a centered normal law. The asymptotic variance of $T_n$ (or $T^\prime_n$) will be explicitly derived for positive stable distribution (or Tweedie distribution) in the following two sections. \end{remark}

	\begin{example}
	An interesting, even if simple, illustration of the suggested methodology is obtained by considering a r.v.\ $X$ with law defined by means of the Laplace transform \eqref{jacobi2}. If $c=\log(e+\sqrt{e^2-1})$, from
	\begin{gather*}
	L_X(a)={{1}\over{{\rm cosh}(a^\gamma)}}=e^{-1},
	\end{gather*}
	we promptly obtain $a_*=c^{1/\gamma}$. Thus, since $k=1$ and $r=0$ in this case, we trivially have
	\begin{gather*}
	\gamma=h_1(a_*)={{\log(c)}\over{\log(a_*)}}
	\end{gather*}
	and an estimator $\widehat\gamma$ of $\gamma$ is given by
	\begin{gather*}
	\widehat\gamma={{\log(c)}\over{\log(A)}}.
	\end{gather*}
	In addition, since it holds 
	\begin{gather*} 
	m_{1,a_*}={{c\,{\rm sinh}(c)\gamma}\over{e^2a_*}},
	\end{gather*}
	a GOF test statistic is given by
	\begin{gather*}
	T_n=\sqrt n\{\widehat m_{1,A}-e^{-2}c\, {\rm sinh}(c)A^{-1}\widehat\gamma\}.
	\end{gather*}
	By using Proposition \ref{ProGen} and the Delta Method, $\sqrt n(\widehat\gamma-\gamma)$ and $T_n$ converge in distribution to  centered normal laws, where the variances can be consistently estimated by substituting $\widehat\gamma$ and $A$ for $\gamma$ and $a_*$ in the corresponding theoretical expressions. 
 	\end{example}

	The results for the distribution family considered in the previous example, depending on a single parameter, are appealing and relatively simple. A more complicated setting occurs for a family depending on two or three parameters, as shown in Section \ref{SecPS} and in Section \ref{SecTW}, even if the results are rather similar in concept. In this more complex framework the following result can be useful.
	
		\begin{remark}\label{rem:mom} If $L_X(s)=g(\lambda\{(\theta+s)^\gamma-\theta^\gamma\})$, after some algebra, we also have
		\begin{gather}\label{mom2}
			E(Y^2_a)=E(Y_a)^2{{g^{\prime\prime}
					(\lambda\{(\theta+a)^\gamma-\theta^\gamma\})}\over{g^\prime
					(\lambda\{(\theta+a)^\gamma-\theta^\gamma\})^2}}
			+E(Y_a){{1-\gamma}\over{\theta+a}},
		\end{gather}
		and 	\begin{gather}\label{mom3}
			\kern-2.7mm E(Y^3_a)
			\kern-0,9mm=\kern-0.9mm {{E(Y_a)^3 g^{\prime\prime\prime}
					(\lambda\{(\theta\kern-0.6mm+\kern-0.6mm a)^\gamma\kern-0.6mm-\kern-0.5mm\theta^\gamma\})}\over{g^\prime(\lambda\{(\theta\kern-0.6mm+\kern-0.6mm a)^\gamma\kern-0.6mm-\kern-0.5mm\theta^\gamma\})^3}}\kern-0,7mm+\kern-0,7mm{{1\kern-0.6mm-\kern-0.6mm\gamma
				}\over{\theta\kern-0.4mm+\kern-0.4mm a}}\Big\{3E(Y^2_a)\kern-0,6mm+\kern-0,6mm E(Y_a)
			\,{{2\gamma\kern-0.4mm-\kern-0.4mm 1}\over{\theta+a}}\Big\}.
		\end{gather}
	\end{remark}

	\section{Inference for the positive stable distribution}\label{SecPS} 

	In this section we assume that $X$ is distributed with a positive stable law. The corresponding Laplace transform \eqref{LapPS} is obtained from class \eqref{class2} by setting $g(x)=e^{-\lambda x}$ and $\theta=0$. Thus, the model involves two parameters $(\gamma,\lambda)$, where the parameter space is $\gamma\in\, ]0,1]$ and $\lambda>0$. It is apparent that $X$ does not have a finite mean for $\gamma<1$. On the basis of \eqref{LapPS} and \eqref{mom1} respectively, for $a>0$ we have
	\begin{gather*}
	L_X(a)=e^{-\lambda a^\gamma},\qquad {\rm and}\qquad E(Y_a)=E(Xe^{-aX})=-{{\gamma}\over{a}}L_X(a)\log(L_X(a)),
	\end{gather*}
	from which we respectively obtain 
	\begin{gather*}
	\gamma=-{{am_{1,a}}\over{L_X(a)\log(L_X(a))}},\qquad {\rm and}\qquad \lambda=-a^{-\gamma}\log(L_X(a)).
	\end{gather*}
	It follows that $$a_*=\lambda^{-1/\gamma}$$ and
	$(\gamma,\lambda)=(h_1(m_{1,a_*},a_*),h_2(m_{1,a_*},a_*))
	=(em_{1,a_*}a_*,a_*^{-\gamma}).$
	Therefore, by using the results provided in Section \ref{SecCenso}, the estimators of $(\gamma,\lambda)$ are given by
	\begin{gather}\label{estpos}
	(\widehat\gamma,\widehat\lambda)=(e\widehat m_{1,A}\,A,\,A^{-\widehat\gamma}).
	\end{gather}
	Properties of estimators \eqref{estpos} are obtained in the following Proposition.

	\begin{proposition}\label{ProPS}
	The estimators $(\widehat\gamma,\widehat\lambda)$ converge almost surely to $(\gamma,\lambda)$ and the random vector $\sqrt n(\widehat\gamma-\gamma,\widehat\lambda-\lambda)$
	converges in distribution to the bivariate normal law $\mathrm{N}_2(0,\Sigma)$ as $n\rightarrow\infty$. Here, $\Sigma$ is the variance-covariance matrix of the random vector
	\begin{gather*}
		(S_1,S_2)=(a_*Xe^{1-a_*X},-\lambda e^{1-a_*X}\{a_*X\log(a_*)+1\})
	\end{gather*}
	A consistent estimator of $\Sigma$  is given by the sample variance-covariance matrix of the random vectors $(\Gamma_i,\Lambda_i)$, with $i=1,\ldots,n$, where 
	\begin{gather*}
		(\Gamma_i,\Lambda_i)=(AX_ie^{1-AX_i},-\widehat\lambda e^{1-AX_i}\{ X_iA\log(A)+1\}). 
	\end{gather*}
	\end{proposition}

	\begin{proof}
	 Since $\gamma=h_1(m_{1,a_*},a_*)$ and $\lambda=h_2(m_{1,a_*},a_*)$, we have the Jacobian matrix of $(h_1,h_2)$ given by
	\begin{equation*}
		J=\begin{pmatrix} ea_*&em_1\\ea_*\lambda\log (a_*) & e\lambda\log (a_*)m_{1,a_*}-a_*^{-1}\gamma\lambda \end{pmatrix}.
	\end{equation*}
	On the basis of Proposition \ref{ProGen} and the Delta Method, $\sqrt n(\widehat\gamma-\gamma,\widehat\lambda-\lambda)$ converges in distribution to $\mathrm{N}_2(0,\Sigma)$, where $\Sigma$ is the variance-covariance matrix of the random vector 
	\begin{gather*}
		(V_1,W)J^{\top}=(e\{a_*V_1+m_{1,a_*}W\},e\lambda\log(a_*)\{a_*V_1+m_{1,a_*}W\}-a_*^{-1}\gamma\lambda W),
	\end{gather*}
	and $V_1$ and $W$ are defined in Proposition \ref{ProGen}. The previous random vector is actually equal to $(S_1,S_2)$, since $m_{1,a_*}=a_*m_{2,a_*}$. Finally, the Proposition is proven, since $A$ converges almost surely to $a_*$ and, by means of the Law of Large Numbers, the sample variance-covariance matrix of the random vectors $(\Gamma_i,\Lambda_i)$ converges almost surely to $\Sigma$. 
	\end{proof}

	Since $m_{1,a_*}=a_*m_{2,a_*}$, by using \eqref{syss}, we have $\widehat L_X^{(2)}(A)\sim \widehat m_{1,A}/A$ as $n\rightarrow\infty$ and expression \eqref{test1} suggests the equivalent GOF test statistic
	\begin{gather}\label{postest}
	T_n=\sqrt n(A\widehat m_{2,A}-\widehat m_{1,A}).
	\end{gather}
	The following Proposition provides the large-sample distribution of \eqref{postest}.

	\begin{proposition}\label{ProgofPS}
	The test statistic $T_n$ in \eqref{postest} converges in distribution to $\mathrm{N}(0,\sigma^2)$ as $n\rightarrow\infty$, where 
	\begin{gather*}
		\sigma^2={\rm Var}\Big(e^{-a_*X}\Big\{{{a_*m_{3,a_*}-2m_{2,a_*}}\over{m_{1,a_*}}}+X(1-a_*X)\Big\}\Big).
	\end{gather*} 
	$\sigma^2$ can be estimated by
	\begin{gather*}
		\widehat\sigma^2={{1}\over{n-1}}\sum_{i=1}^n (Z_i-\overline{Z})^2,
	\end{gather*}
	where
	\begin{gather*}
		Z_i=e^{-AX_i}\Big\{{{A\widehat m_{3,A}-2\widehat m_{2,A}}
			\over{\widehat m_{1,A}}}+
		X_i(1-AX_i)\Big\},\qquad {\rm and}\qquad
		\overline Z={{1}\over{n}}\sum_{i=1}^n Z_i.
	\end{gather*}
	\end{proposition}

	\begin{proof}
	Since
	\begin{gather*}
		T_n=\sqrt n\{(A\widehat m_{2,A}-\widehat m_{1,A})-(am_{2,a}-m_{1,a})\},
	\end{gather*}
	on the basis of Proposition \ref{ProGen} and the Delta Method, $T_n$ converges in distribution to the normal law $\mathrm{N}(0,\sigma^2)$, where $\sigma^2$ is the variance of the r.v.\ $(-V_1+aV_2+m_{2,a}W)$, where $V_1$, $V_2$ and $W$ are defined in Proposition \ref{ProGen}. It holds
	\begin{gather*}
		-V_1+a_*V_2+m_{2,a_*}W=e^{-a_*X}\Big\{-X(1-a_*X)-{{a_*m_{3,a_*}}\over{m_{1,a_*}}}+2{{m_{2,a_*}}\over{m_{1,a_*}}}\Big\},
	\end{gather*}
	and the Proposition is proven, since $A$ converges almost surely to $a_*$ and, by means of the Law of Large Numbers, the sample variance $\widehat\sigma^2$ converges almost surely to $\sigma^2$. 
	\end{proof}

	\section{Inference for the Tweedie distribution}\label{SecTW}

	In this section we assume that $X$ is distributed with the Tweedie law. The corresponding Laplace transform \eqref{LapTW} is a special case of the class \eqref{class2} by setting
	\begin{gather*}
	g(x)=e^{-\mathrm{sgn}(\gamma)\lambda x}.
	\end{gather*}
	As to the Tweedie family, the parameter space is more complex with respect to the positive stable case. Indeed, the parameters $(\gamma,\lambda,\theta)$ are defined on the union of $]0,1]\times\,]0,\infty[\,\times[0,\infty[$ and $]-\infty,0[\times\,]0,\infty[\,\times\,]0,\infty[$. On one hand, when the parameters are elements of the first set, the Tweedie distribution is actually an exponentially-tilted positive stable law. On the other hand, when the parameters are elements of the second set, the Tweedie distribution is a compound Poisson law (see \citealp{AAL}). 
	
Thus, the extra parameter $\theta$ substantially extends the flexibility of the Tweedie model with respect to the positive stable model. For further properties of the Tweedie law, see \cite{BARCE}, \cite{BARCEPE}, and references therein.

 If $\lambda\theta^\gamma>-{\rm sgn}(\gamma)$, it holds
	$$a_*=\Big\{\frac{1}{{\rm sgn}(\gamma)\lambda}+\theta^\gamma\Big\}^{1/\gamma}-\theta.$$  Note that $P(X=0)<1/e$ $\iff$ $ \lambda\theta^\gamma>-{\rm sgn}(\gamma)$.  From \eqref{mom1}, \eqref{mom2} and \eqref{mom3} the following expressions arise 
	\begin{gather*}
	E(Y_{a_*})=|\gamma|\lambda L_X(a_*)(\theta+a_*)^{\gamma-1},\qquad	E(Y^2_{a_*}) ={{E(Y_{a_*})^2}\over{L_X(a_*)}}+E(Y_{a_*}){{1-\gamma}\over{\theta+a_*}},
	\end{gather*}
	\begin{gather*}
		E(Y^3_{a_*}) ={{E(Y_{a_*})^3}\over{L^2_X(a_*)}}+{{E(Y_{a_*})(1-\gamma)}\over{\theta+a_*}}\Big\{{{3E(Y_{a_*})}\over{L_X(a_*)}}+{{2-\gamma}\over{\theta+a_*}}\Big\}.
	\end{gather*}
	
\medskip\noindent Moreover, 
$$\gamma= h_1(m_{1,a_*}, m_{2,a_*}, m_{3,a_*}, a_*)= 1- \Big({{m_{2,a_*}}\over{m_{1,a_*}^2}}-e\Big)\phi (m_{1,a_*}, m_{2,a_*},m_{3,a_*}),$$
 $$	\lambda=h_2(m_{1,a_*},m_{2,a_*},m_{3,a_*},a_*)=e m_{1,a_*}|\gamma|^{-1}(\theta+a_*)^{1-\gamma},$$
 and $$\theta = h_3(m_{1,a_*},m_{2,a_*},m_{3,a_*},a_*)=-a_*+m_{1,a_*}^{-1}\phi(m_{1,a_*},m_{2,a_*},m_{3,a_*}),$$ 
 where $$\phi(m_{1,a_*},m_{2,a_*},m_{3,a_*})=\Big({{m_{3,a_*}-e^{2}m_{1,a_*}^3}\over{m_{1,a_*}m_{2,a_*}-e m_{1,a_*}^3}}-2e-{{m_{2,a_*}}\over{m_{1,a_*}^2}}\Big)^{-1}.$$

\smallskip\noindent Therefore, on the basis of the results given in Section \ref{SecCenso}, the estimators $(\widehat\gamma,\widehat\lambda, \widehat\theta)$ for the parameters $(\gamma,\lambda,\theta)$ are given by 
\begin{align}
\label{tweediestime1}
\widehat\gamma&=1-\Big({{\widehat m_{2,A}}\over{\widehat m_{1,A}^2}}-e\Big)\phi(\widehat m_{1,A},\widehat m_{2,A},\widehat m_{3,A}),\\
\label{tweediestime2}
\widehat\lambda&=e\widehat m_{1,A}|\widehat\gamma|^{-1}(\widehat\theta+A)^{1-\widehat\gamma},\\
\label{tweediestime3}
\widehat\theta&=-A+\widehat m_{1,A}^{-1}\, \phi(\widehat m_{1,A},\widehat m_{2,A},\widehat m_{3,A}).
\end{align}

\medskip The large-sample properties of the previous estimators are given in the following Proposition.
\begin{proposition}\label{ProTW}
	If $\lambda\theta^\gamma>-{\rm sgn}(\gamma)$, the estimators $(\widehat\gamma,\widehat\lambda,\widehat\theta)$ converges almost surely to $(\gamma,\lambda,\theta)$ and $\sqrt n(\widehat\gamma-\gamma,\widehat\lambda-\lambda,\widehat\theta-\theta)$ converges in distribution to the trivariate normal law $\mathrm{N}_3(0,\Sigma)$ as $n\rightarrow\infty$. Here, $\Sigma$ is the variance-covariance matrix of the random vector
	\begin{gather*}
		(S_1,S_2,S_3)=e^{-a_*X}\Big(X-{{m_{2,a*}}\over{m_{1,a_*}}},X^2-{{m_{3,a_*}}\over{m_{1,a_*}}},X^3-{{m_{4,a_*}}\over{m_{1,a_*}}},{{1}\over{m_{1,a_*}}}\Big)J^{\top}
	\end{gather*}
	and $J$ represents the Jacobian matrix of order $3\times4$ of the map $h=(h_1,h_2,h_3)$ evaluated at $(m_{1,a_*},m_{2,a_*},m_{3,a_*},a_*)$. In particular, an estimator of $\Sigma$  is given by the sample variance-covariance matrix of $(\Gamma_i,\Lambda_i,\Theta_i)$ with $i=1,\ldots,n$, where
	\begin{gather*}
		(\Gamma_i,\Lambda_i,\Theta_i)=e^{-AX_i}\Big(X_i-{{\widehat m_{2,A}}\over{\widehat m_{1,A}}},X^2_i-{{\widehat m_{3,A}}\over{\widehat m_{1,A}}},X^3_i-{{\widehat m_{4,A}}\over{\widehat m_{1,A}}},\,{{1}\over{\widehat m_{1,A}}}\Big){\widehat J}^{\top}
	\end{gather*}
	and ${\widehat J}=J(\widehat m_{1,A},\widehat m_{2,A},\widehat m_{3,A},A)$.
	\end{proposition}

	\begin{proof}
	By means of Proposition \ref{ProGen} and the Delta Method applied to $h$, we have the convergence in distribution of $\sqrt n(\widehat\gamma-\gamma,\widehat\lambda-\lambda,\widehat\theta-\theta)$ to the trivariate normal law $\mathrm{N}_3(0,\Sigma)$, where $\Sigma$ is the variance-covariance matrix of the random vector $(S_1,S_2,S_3)$. Moreover, $(\widehat\gamma,\widehat\lambda,\widehat\theta)$ is a consistent estimator of $(\gamma,\lambda,\theta)$ by means of the Law of Large Numbers and owing to the continuity of $h$. Since $A$ converges almost surely to $a_*$, the sample variance-covariance matrix of $(\Gamma_i,\Lambda_i,\Theta_i)$ converges almost surely to $\Sigma$ and the Proposition is proven. 
	\end{proof}
On the basis of \eqref{test2}, let us consider the test statistic
	\begin{gather*}
	T^\prime_n=\sqrt n[{\rm sgn}(\widehat\gamma)\widehat\lambda\{(\widehat\theta+A)^
	{\widehat\gamma}-
	\widehat\theta^{\,\widehat\gamma}\}-1].
	\end{gather*}
	The previous test statistic, suitably normalized, has the same asymptotic behaviour of the normalized version of
	\begin{gather}\label{twtest}
	\widetilde T^\prime_n=\sqrt n\Big\{1- \Big({{\widehat\theta}\over{\widehat\theta+A}}\Big)^
	{\widehat\gamma}-{{\widehat\gamma}
		\over{e\widehat m_{1,A}(\widehat\theta+A)}}
	\Big\}.
	\end{gather}
	The large-sample properties of the test statistic $\widetilde T^\prime_n$ in \eqref{twtest} are derived in the following Proposition.

	\begin{proposition}\label{ProgofTW}
	If $\lambda\theta^\gamma>-{\rm sgn}(\gamma)$, the test statistic $\widetilde T^\prime_n$ in \eqref{twtest} converges in distribution to the normal law $\mathrm{N}(0,\sigma^2)$ as $n\rightarrow\infty$, where
	\begin{gather*}
		\sigma^2={\rm Var}(\beta_1V_1+\beta_2V_2+\beta_3V_3+\beta_4W)
	\end{gather*}
	and $V_1$, $V_2$, $V_3$ and $W$ are defined in Proposition \ref{ProGen}. Moreover, $(\beta_1,\beta_2,\beta_3,\beta_4)$ is the gradient of the map 
	$$
	(m_1,m_2,m_3,a)\mapsto -(1-am_{1}\psi)^
		{\phi}-e^{-1}\Big(\psi-{{m_2}\over{m_1^2}}\Big)
	$$
	evaluated at $(m_{1,a_*},m_{2,a_*},m_{3,a_*},a_*)$, where
	\begin{gather*}
		\psi=\psi(m_1,m_2,m_3)={{m_{3}-e^{2}m_{1}^3}\over{m_{1}m_{2}-em_{1}^3}}-2e-{{ m_{2}}\over{m_{1}^2}}, \quad \phi=\phi(m_1,m_2,m_3)=1-{{1}\over{\psi}}\Big({{m_{2}}\over{ m_{1}^2}}-e\Big).
	\end{gather*}
	In particular, 
	\begin{gather*}
		\widehat\sigma^2={{1}\over{n-1}}\sum_{i=1}^n (Z_i-\overline{Z})^2
	\end{gather*}
	is a consistent estimator for $\sigma^2$, where $		Z_i=\widehat\beta_1 \widehat V_{1,i}+\widehat\beta_2\widehat V_{2,i}+\widehat\beta_3\widehat V_{3,i}+\widehat\beta_4\widehat W_{i}$ and
	$\overline Z={{1}\over{n}}\sum_{i=1}^n Z_i,$ while $\widehat\beta_r$ denotes $\beta_r$ evaluated at $(\widehat m_{1,A},\widehat m_{2,A},\widehat m_{3,A},A)$, with $r=1,\ldots,4$, and $\widehat V_{1,i}$, $\widehat V_{2,i}$, $\widehat V_{3,i}$ and $\widehat W_{i}$ are defined in Proposition \ref{ProGen}.
	\end{proposition}

	\begin{proof}
	Since
	\begin{gather*}
		\widetilde T^\prime_n=\sqrt n\Big\{-(1-A\widehat m_{1,A}\widehat\psi)^
		{\widehat\phi}-e^{-1}\Big(\widehat\psi-{{\widehat m_{2,A}}\over{\widehat m_{1,A}^2}}\Big)\Big\}
	\end{gather*}
	with 
	\begin{gather*}
		\widehat\psi={{\widehat m_{3,A}-e^{2}\widehat m_{1,A}^3}\over{\widehat m_{1,A}\widehat m_{2,A}-e\widehat m_{1,A}^3}}-2e-{{\widehat m_{2,A}}\over{\widehat m_{1,A}^2}}, \ \quad \ \widehat\phi=1-{{1}\over{\widehat\psi}}\Big({{\widehat m_{2,A}}\over{\widehat m_{1,A}^2}}-e\Big),
	\end{gather*}
	on the basis of Proposition \ref{ProGen} and of the Delta Method, by considering the equality
	\begin{gather*}
		(1-a_*m_{1,a_*}\psi)^{\phi}=-e^{-1}\Big(\psi-{{ m_{2,a_*}}\over{ m_{1,a_*}^2}}\Big),
	\end{gather*}
	$\widetilde T^\prime_n$ converges in distribution to the normal law $\mathrm{N}(0,\sigma^2)$. The result follows from Proposition \ref{ProTW} and the Law of Large Numbers.
	\end{proof}

\section{Simulation results}\label{SecSimu}
By means of a Monte Carlo simulation study, we assess the finite sample performance of the estimators and GOF procedures proposed in Section \ref{SecPS} for the positive stable PS$(\gamma,\lambda)$ distribution and in Section \ref{SecTW} for the Tweedie TW$(\gamma,\lambda,\theta)$ distribution. 
For the Tweedie distribution, when it is  a compound Poisson law, that is when $\gamma < 0$, it can be useful to determine the parameters values characterizing the distributions from which data are generated using the following notation
\begin{gather}
	\label{alternative_param}
	{\rm TW}(\gamma,\lambda, \theta)= {\rm TW}_0(|\gamma|\lambda \theta^{\gamma-1},\frac{1-\gamma}{\theta},\exp(-\lambda \theta^\gamma))={\rm TW}_0(\mu,w,p).
\end{gather}
 It should be noted that $\mu$ and  $p$ have an immediate interpretation as they represent the mean and the probability of observing a zero value,  respectively. Then, if $\gamma <0$, the model parameters are $(\gamma,\lambda, \theta)$, but their values are derived from the values selected for $\mu$, $w$ and $p$.
The random variate generation is carried out by considering the stochastic representations discussed in \cite{DEJAM} and \cite{BARCEPE}.

As to the positive stable distribution, first we consider four choices of the parameter values and, for each choice, $3500$ samples of size $n=100,200,300$ are independently generated. For each sample, parameter estimates are obtained by means of \eqref{estpos}. From the Monte Carlo distributions of the estimates, the Relative Root Mean Squared Error (RRMSE) of $\widehat{\gamma}$ and $\widehat{\lambda}$ is computed as the root mean square error divided by the true parameter value.
Analogously, two parameter choices are considered for each parametrization of the Tweedie distribution and, for each choice, $3500$ samples of size $n=500,1000, 1500$ are independently generated. The different sample sizes considered for the two distributions are related to the different number of parameters.
For each sample the estimates of $\gamma$, $\lambda$ and $\theta$ are obtained by means of \eqref{tweediestime1}, \eqref{tweediestime2} and \eqref{tweediestime3}, and the RRMSE values are computed from the corresponding Monte Carlo distributions.
Percentage values of RRMSEs are reported in Table \ref{tab:rrmse_ps} and in Table \ref{tab:rrmse_tw} for the positive stable law and for the Tweedie law, respectively. In these Tables and in all subsequent ones, subscripts denote the sample size.
The results illustrate the consistency of the estimators for all the considered parameter values. 

Moreover, for the positive stable distribution, also confidence interval estimates at confidence level $1-\alpha=0.95$ are obtained using the quantiles of the standard normal distribution and results in Section \ref{SecPS}. The empirical coverage is computed as the proportion of  times the true value is in the interval.
More precisely, for any fixed value $\gamma=0.3, 0.5, 0.7, 0.8$ and for $\lambda$ varying from $0.5$ to $12$ by $0.5$, $3500$ independent samples of $n=100, 200$ are generated and the empirical coverages of the 0.95 confidence intervals are computed for $\gamma$ and $\lambda$ and depicted in Figure \ref{cov_gamma} and Figure \ref{cov_lambda}, respectively.
The figures  highlight that the empirical coverage of the interval estimators for both parameters attains the nominal level of $0.95$ for moderately large sample sizes.

	\begin{table}[!ht]
		\caption{Percentage values of RRMSE of $(\widehat \gamma,\widehat \lambda)$ for the positive stable distribution for some parameter choices.}
		
		\medskip
		\centering
		\small
		\begin{tabular}{lrrrrrr}
			\hline
			Model &  $\widehat{\gamma}_{100}$ & $ \widehat \lambda_{100}$ & $\widehat{\gamma}_{200}$ & $ \widehat \lambda_{200}$ & $ \widehat \gamma_{300}$ & $\widehat \lambda_{300}$  \\
			\hline
			$\mathrm{PS}(0.3,2)$   			& 11.84 & 12.46 & 8.34 &  8.49 & 6.77 &  6.88\\							
			$\mathrm{PS}(0.4,5)$   			&  9.19 & 15.44 & 6.50 & 10.57 & 5.30 &  8.61\\
			$\mathrm{PS}(0.5,15)$  			&  7.31 & 18.81 & 5.19 & 12.96 & 4.22 & 10.54\\
			$\mathrm{PS}(0.6,20)$  			&  5.88 & 15.70 & 4.15 & 10.87 & 3.38 &  8.88\\
			\hline
		\end{tabular}
		
		\label{tab:rrmse_ps}
	\end{table}

	\begin{table}[!ht]
		\caption{Percentage values of RRMSE of $(\widehat \gamma,\widehat \lambda,\widehat \theta)$ for the Tweedie distribution for some parameter choices.}
		
		\medskip
		\centering
	    \resizebox{\textwidth}{!}{
		
		\begin{tabular}{lrrrrrrrrr}
			\hline
				Model & $\widehat\gamma_{500}$ & $\widehat\lambda_{500}$ & $\widehat\theta_{500}$ & $\widehat\gamma_{1000}$ & $\widehat\lambda_{1000}$ & $\widehat\theta_{1000}$ & $\widehat\gamma_{1500}$ & $\widehat\lambda_{1500}$ & $\widehat\theta_{1500}$ \\
			\hline
			$\mathrm{TW}_0(1,1,0.1)$   & 28.66 & 37.49 & 23.13 & 19.56 & 19.73 & 15.80 & 16.02 & 15.08 & 13.12  \\
			$\mathrm{TW}_0(1,1.25,0.2)$  & 22.61 & 48.19 & 26.53 & 15.76 & 27.10 & 18.25 & 12.73 & 20.44 & 14.89  \\
			$\mathrm{TW}(0.5,2,0.5)$     & 9.54 & 18.39 & 24.84 &  6.73 & 12.14 & 17.62 &  5.46 &  9.72 & 14.19  \\
			$\mathrm{TW}(0.6,2.5,0.6)$   & 7.04 & 13.94 & 21.34 &  4.83 &  8.91 & 14.38 &  3.89 &  7.09 & 11.69  \\
			\hline
		\end{tabular}
     	}
		
		\label{tab:rrmse_tw}
	\end{table}
	
	\begin{figure}[!ht]
		\centering
		\includegraphics[width=0.9\textwidth]{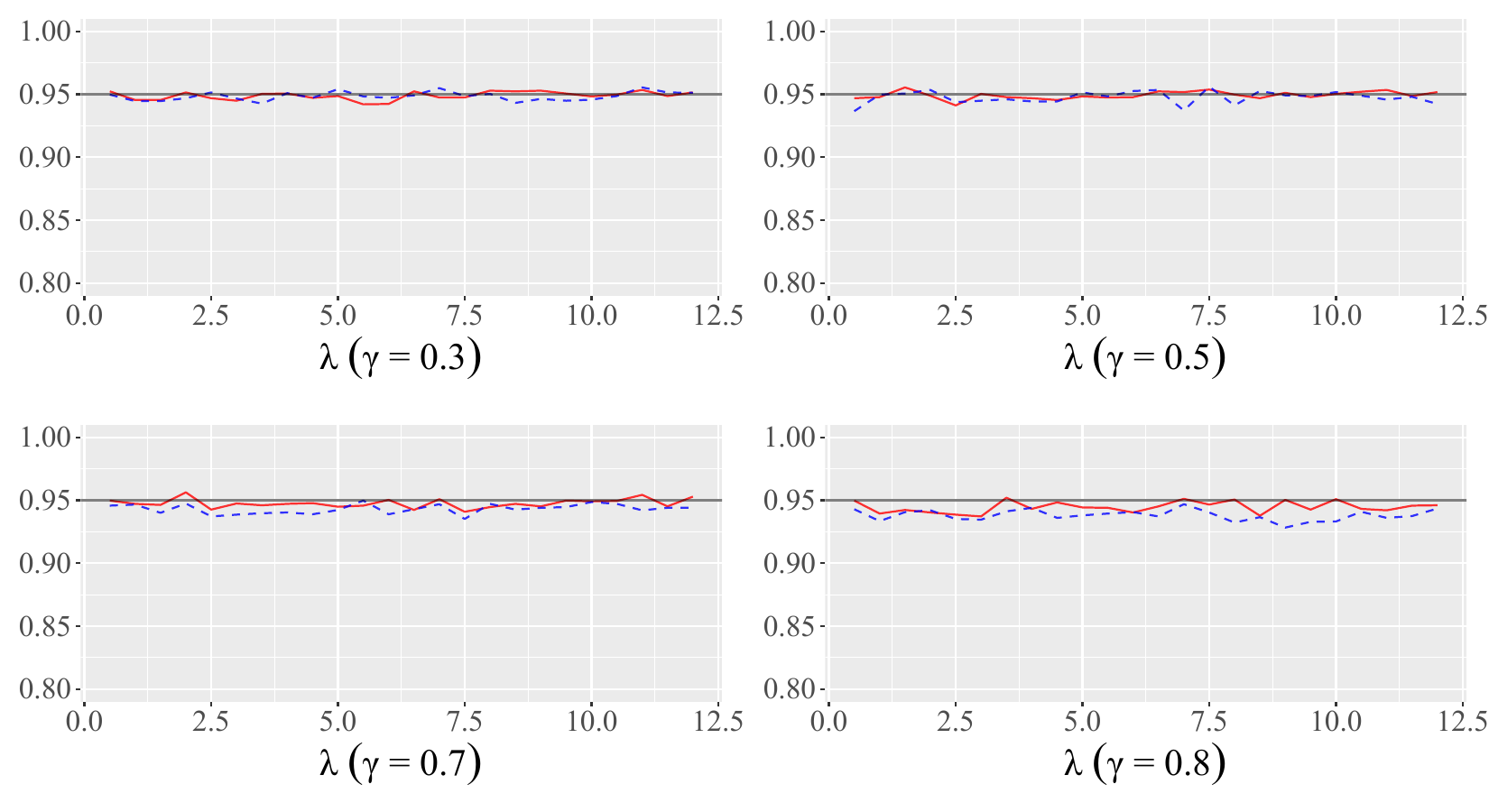}
		\caption{Empirical coverage of the 0.95 confidence interval for $\gamma$
 in positive stable distributions for different choices
			of 
$\gamma$ and $\lambda$. Dashed line for $n=100$ and solid line for $n=200$.}
		\label{cov_gamma}
	\end{figure}

	\begin{figure}[!ht]
		\centering
		\includegraphics[width=0.9\textwidth]{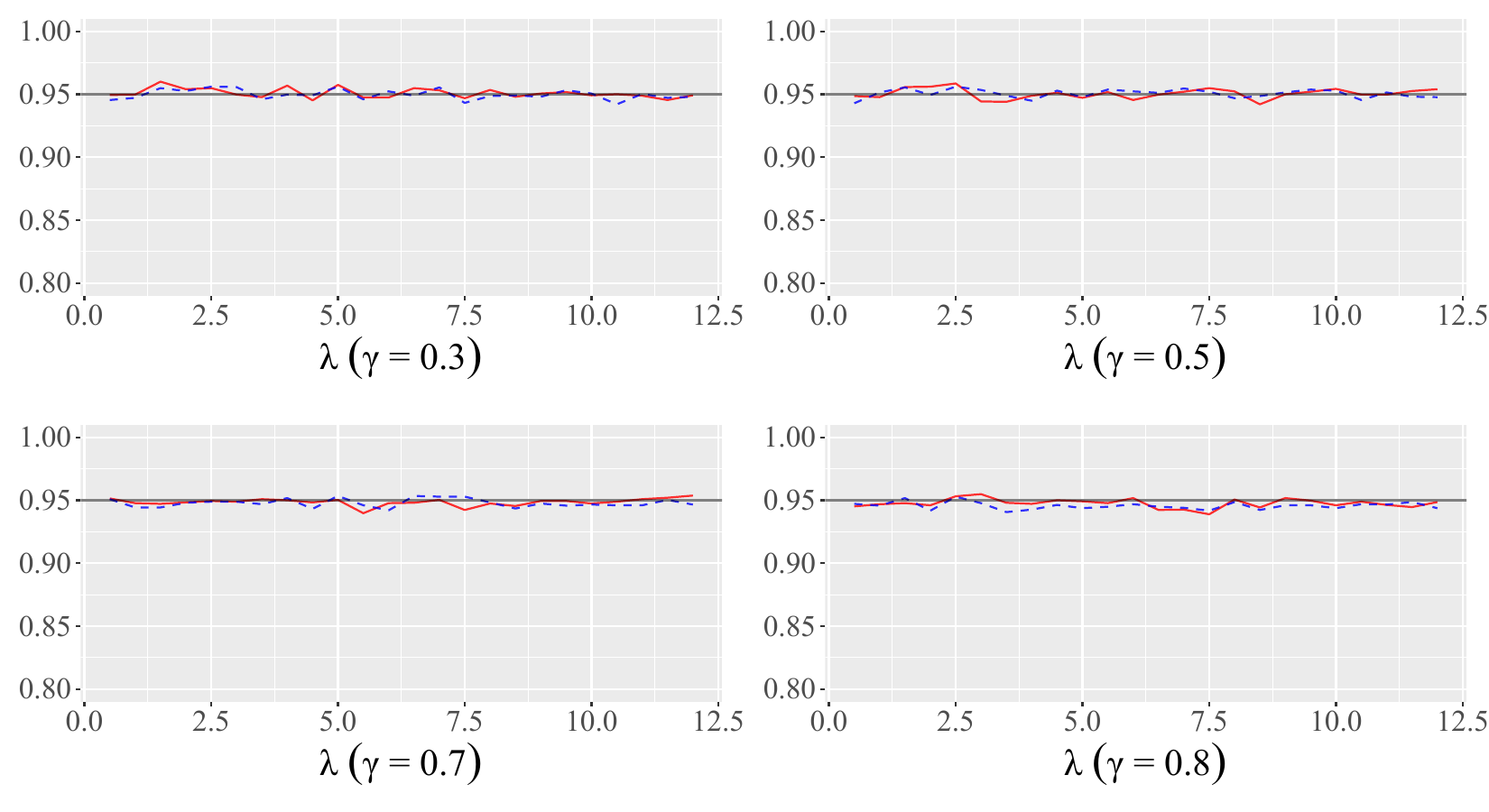}
		\caption{Empirical coverage of the 0.95 confidence interval for $\lambda$
 in positive stable distributions for different choices
			of 
$\gamma$ and $\lambda$. Dashed line for $n=100$ and solid line for $n=200$.}
		\label{cov_lambda}
	\end{figure}

Concerning the GOF, both the testing procedures based on \eqref{postest} and \eqref{twtest} are evaluated in terms of the empirical significance and power levels. First,  $3500$ independent samples are generated under the null hypothesis $H_0$, which represents the functional hypothesis that the target r.v.\ is distributed according to the positive stable distribution or the Tweedie distribution. 
Then, the empirical significance level is computed as the proportion of rejections at the nominal level equal to $0.05$.

Under the positive stable model, we obtain the empirical significance level for various representative parameter choices and for samples of size $n=100,200,300$ (see Table \ref{tab:sig_level_ps}). The empirical significance level is also depicted in Figure \ref{fig:sig_level_ps} for $\gamma=0.3,0.5,0.7,0.8$ and $\lambda$ varying from $0.5$ to $12$ by $0.5$, for $n=100,200$. From Table \ref{tab:sig_level_ps} and Figure \ref{fig:sig_level_ps}, it is apparent that the test shows a satisfactory significance level, approaching the nominal level as $n$ increases, albeit it is slightly conservative for small $\gamma$ when the sample size is small. Analogously, under the Tweedie model, we compute the empirical significance level for some parameter choices for $n=500,1000,1500$. The parameter values when $\gamma<0$ are given in Table \ref{tab:param_conversion}. The results are reported in Table \ref{tab:sig_level_tw}. As it could be expected, the testing problem is much more challenging for the Tweedie GOF test owing to the additional parameter. Indeed, the test achieves the nominal level for a larger sample size such as $n=1500$. Anyway, it should be remarked that the test is conservative.

	\begin{table}[!ht]
		
		\centering
		\caption{Percentage values of the empirical significance level (nominal level 0.05) of the positive stable GOF test for some parameter choices.}
		
		\medskip
		\label{tab:sig_level_ps}
		\small
		\begin{tabular}{lrrr}
			\hline
			Model &$T_{100}$ & $ T_{200}$ & $T_{300}$ \\
			\hline
			$\mathrm{PS}(0.3,2)$     & 2.83  & 3.94  & 4.14\\
			$\mathrm{PS}(0.4,5)$     & 3.49  & 3.89  & 4.34\\
			$\mathrm{PS}(0.5,15)$    & 3.69  & 4.74  & 5.20\\
			$\mathrm{PS}(0.6,20)$    & 3.97  & 4.54  & 4.89\\
			\hline
		\end{tabular}
	\end{table}
	
	\begin{figure}[!ht]
		
		\centering
		\includegraphics[width=0.9\textwidth]{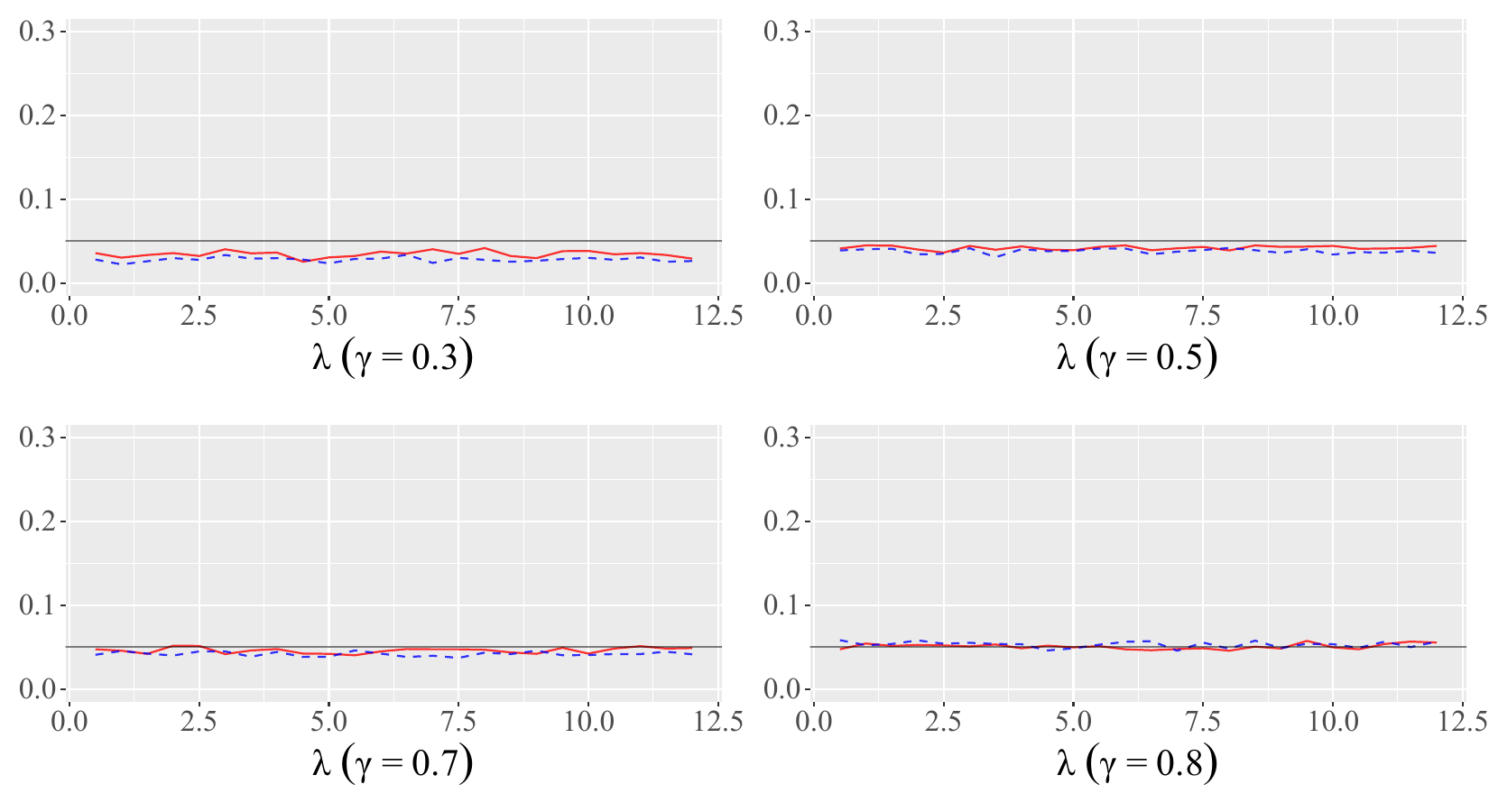}
		\caption{Empirical significance level (nominal level 0.05) for different choices
			of 
$\gamma$ and $\lambda$.  Dashed line for $n=100$ and solid line for $n=200$.}
		\label{fig:sig_level_ps}
	\end{figure}

	\begin{table}[!ht]
		
		\centering
		\caption{Percentage values of the empirical significance level (nominal level 0.05) of the Tweedie GOF test for some parameter choices.}
		
		\medskip
		\label{tab:sig_level_tw}
		\small
		\begin{tabular}{lrrrr}
			\hline
			Model &$\widetilde T_{300}^\prime$ & $\widetilde T_{500}^\prime$ & $\widetilde T_{1000}^\prime$ & $\widetilde T_{1500}^\prime$ \\
			\hline
			$\mathrm{TW}_0(0.75,0.5,0.1)$  & 1.17  & 2.11  & 3.31   & 3.43\\
			$\mathrm{TW}_0(1,1,0.1)$       & 1.17  & 2.43  & 3.14   & 3.49\\
			$\mathrm{TW}_0(1,1.25,0.2)$    & 1.09  & 2.03  & 2.66   & 3.20\\
			
			$\mathrm{TW}(0.5,2,0.5)$         & 1.06 & 1.83 & 3.77 & 3.80\\
			$\mathrm{TW}(0.6,2.5,0.6)$      & 0.17 & 1.49 & 3.26 & 3.40\\
			\hline
		\end{tabular}
	\end{table}

	As to the power of the two tests, we consider as alternative distributions the Linnik law denoted by $\mathrm{LI}(\gamma,\lambda,\delta)$, the shape-scale Pareto law denoted by $\mathrm{PA}(\alpha,\beta)$ with density function $f(x)=\alpha \beta^\alpha / x^{\alpha+1}\mathbf{1}_{\{x\geq \beta\}}$, and the shape-scale Weibull law denoted by $\mathrm{WE}(k,\lambda)$ with density function $f(x)=(k/ \lambda) (x/\lambda)^{k-1} \exp (-(x/\lambda)^k)$.
	Moreover, we consider the log-normal law denoted by $\mathrm{LN}(\mu,\sigma)$ and, as a further alternative, $\mathrm{LN}^{1/2}(\mu,\sigma)$, that represent respectively the laws of the r.v.s $Y=e^{X}$ and $Y=e^{X^2}$, where $X$ has normal law $\mathrm{N}(\mu,\sigma)$. We also consider some zero-inflated distributions, where the third parameter represents the probability of observing a zero value. Such laws are indicated by adopting a null subscript, e.g.\ $\mathrm{PA}_0(\alpha,\beta,p)$.
	The empirical powers of the tests, computed as the proportion of rejections when data are generated according to alternative distributions, are reported in Table \ref{tab:power_ps} and Table \ref{tab:power_tw} under the null hypothesis of positive stable or Tweedie distributions, respectively. 
	In particular, for each alternative distribution, we generate $3500$ independent samples of size $n=100, 200, 300$ when performing GOF test at significance level $0.05$ for the positive stable distribution and of size $n=300,500, 1000, 1500$ for the Tweedie distribution.
	
	 From Table \ref{tab:power_ps}, the positive stable GOF show an appreciable power for all the considered alternatives. In particular, it achieves an elevate power for the $\mathrm{PA}$ and $\mathrm{LN}^{1/2}$ distribution. In contrast $\mathrm{LI}$ is similar to $\mathrm{PS}$, justifying a smaller power.
	The results in Table \ref{tab:power_tw} evidence that the Tweedie GOF test is consistent, even if the increase in the number of parameters leads to a more challenging testing problem in terms of type 2 error. Indeed, larger samples are necessary to achieve a power comparable to the positive stable case, even if the power is satisfactory for $n$ larger than $500$. Moreover, under zero-inflated models, the procedure achieves an increasing performance as the zero probability increases.

	\begin{table}[!ht]
		\centering
		\caption{Percentage values of the empirical power of the positive stable GOF test for some selected alternative distributions.}
		
		\medskip
		\label{tab:power_ps}
			\small
		\begin{tabular}{lrrr}
			\hline
			Model &$T_{100}$ & $ T_{200}$ & $T_{300}$ \\
			\hline
			$\mathrm{LN}(0,1.5)$           & 63.80  &  95.66   &  99.69\\
			$\mathrm{PA}(5,2)$            & 97.77  &  99.43   &  99.97\\
			$\mathrm{PA}(10,2)$           & 99.97  & 100.00   &  99.97\\
			$\mathrm{LI}(0.5,2,0.5)$      & 10.03  &  27.43   &  43.91\\
			$\mathrm{LI}(0.5,2,0.75)$         & 10.51  &  31.51   &  52.14\\
			$\mathrm{LN}^{1/2}(0,1.5)$         & 87.97  &  99.29   & 100.00\\
			$\mathrm{LN}^{1/2}(0,3)$        & 51.09  &  81.57   &  93.11\\
			\hline
		\end{tabular}
	\end{table}
	\begin{table}[!ht]
		\centering
		\caption{Percentage values of the empirical power of the Tweedie GOF test for some selected alternative distributions.}
		
		\medskip
		\label{tab:power_tw}
			\small
		\begin{tabular}{lrrrr}
			\hline
			Model &$\widetilde T_{300}^\prime$ & $\widetilde T_{500}^\prime$ & $\widetilde T_{1000}^\prime$ & $\widetilde T_{1500}^\prime$ \\
			\hline
			$\mathrm{LN}(0,1)$           & 17.66  &  34.60  &  56.31   &  68.91\\
			$\mathrm{WE}(5,1)$           & 20.46  &  42.03  &  76.00   &  90.57\\
			$\mathrm{PA}(10,2)$          & 15.03  &  26.60  &  52.06   &  61.49\\
			
			$\mathrm{LN}_0(1,0.75,0.1)$    & 69.74  &  96.91  & 100.00   & 100.00\\
			$\mathrm{LN}_0(1,0.75,0.2)$    & 75.17  &  96.57  & 100.00   & 100.00\\
			$\mathrm{LN}_0(5,1,0.1)$       & 27.51  &  58.43  &  93.51   &  99.11\\
			$\mathrm{LN}_0(5,1,0.2)$       & 77.46  &  98.86  & 100.00   & 100.00\\
			$\mathrm{WE}_0(3,1,0.1)$       & 74.11  &  99.57  & 100.00   & 100.00\\
			$\mathrm{WE}_0(3,1,0.2)$       & 78.69  &  99.97  & 100.00   & 100.00\\
			$\mathrm{WE}_0(5,1,0.1)$       & 99.83  & 100.00  & 100.00   & 100.00\\
			$\mathrm{WE}_0(5,1,0.2)$       & 93.14  & 100.00  & 100.00   & 100.00\\
			$\mathrm{PA}_0(5,2,0.1)$       & 20.77  &  41.49  &  69.29   &  85.49\\
			$\mathrm{PA}_0(5,2,0.2)$       & 65.26  &  99.37  & 100.00   & 100.00\\
			$\mathrm{PA}_0(10,2,0.1)$      & 99.86  & 100.00  & 100.00   & 100.00\\

			\hline
		\end{tabular}
	\end{table}
	
	\begin{table}[!ht]
		
		\centering
		\caption{Correspondence between the values of $(\gamma,\lambda,\theta)$ and $(\mu,w,p)$ for the Tweedie distribution from \eqref{alternative_param}.}
		
		\medskip
		\label{tab:param_conversion}
			\small
		\begin{tabular}{lccc}
			\hline
			Model &$\gamma$ & $ \lambda$ & $\theta$ \\
			\hline
			$\mathrm{TW}_0(0.75,0.5,0.1)$   & -1.8689607  & 60.297348  &  5.737921\\
			$\mathrm{TW}_0(1,1,0.1)$        & -0.7677042  &  3.565768  &  1.767704\\
			$\mathrm{TW}_0(1,1.25,0.2)$     & -0.9883402  &  2.546270  &  1.590672\\
			\hline
		\end{tabular}
	\end{table}

	\section{Discussion and concluding remarks}\label{discu}

	We propose a new general inferential approach based on the Laplace transform of a random variable, that relies on exponential random censoring. This approach is useful to derive tractable parameter estimators that are used to propose suitable GOF tests for some non-standard families of random variables. The proposed estimators and test statistics are proven to be asymptotically normal, leading to computationally convenient inferential procedures. We show how exponential censoring can be used to build parameter estimators and GOF tests for the positive stable and for the Tweedie laws for which specific tests are not present in the literature. The simulation study shows that the finite-sample performance of the proposed procedures is rather satisfactory even in the three-parameter case for a moderately large sample size.

\bibliographystyle{apalike}      
\bibliography{bib_Laplace}   
\end{document}